\documentclass[12pt]{article}
\usepackage[utf8]{inputenc}
\usepackage{amsmath, amssymb, amsthm,xcolor}
\usepackage{geometry}
\usepackage{enumerate}
\usepackage{cite}
\usepackage{graphicx}
\newtheorem{thm}{Theorem}[section]
\newtheorem{lemma}[thm]{Lemma}
\newtheorem{corollary}[thm]{Corollary}

\newcommand{\abs}[1]{\left|#1\right|}
\newcommand{\pa}[1]{\left(#1\right)}
\newcommand{\br}[1]{\left[#1\right]}
\newcommand{\cb}[1]{\left\{#1\right\}}

\newcommand{\prescript}[3]{\phantom{}_{#1}{#2}_{#3}}

\title{Reciprocal Polynomials with Zeros on the Unit Circle and Derivatives of Chebyshev Polynomials of the Second Kind}
\author{Dmitriy Dmitrishin, Daniel Gray and Alexander Stokolos}

\begin{document}
\date{}
\maketitle

\begin{center}
\textit{In Memoriam: We dedicate this manuscript to Konstantin Oskolkov, whose 80th birthday would have been  
February 17th, 2026.}
\end{center}

\begin{abstract}
In this article, we consider the reciprocal antisymmetric polynomial
\begin{displaymath}
P(z) = \sum_{j = 0}^{s}(-1)^j\gamma_j\pa{z^j - z^{N + s + 1 - j}}, \ \gamma_0 = 1.
\end{displaymath}
It is shown that if all the zeros of $P(z)$ are located on the unit circle, that
$\displaystyle\abs{\gamma_j} \leq {s \choose j}\pa{{N + s + 1 \choose j}}^{-1}$, $j = 1,\ldots,s$; moreover, these estimates cannot be improved in the general case. Factorization formulas for extremal polynomials are given:
\begin{eqnarray*}
\lefteqn{\sum_{j = 0}^{s}(-1)^j{s \choose j}\pa{{N + s + 1 \choose j}}^{-1}\pa{z^j - z^{N + s + 1 - j}}} \\
&=& (1 - z)^{2s + 1} \prod_{j = 1}^{\br{\frac{N - s}{2}}} \br{z^2 + 1 + 2z\pa{1 - 2(\nu_j)^2}}
\begin{cases}
(1 + z), & N - s \mbox{ is odd} \\
1, & N - s \mbox{ is even}
\end{cases}
\end{eqnarray*}
where
$\cb{\nu_j}_{j = 1}^{\br{\frac{N - s}{2}}}$
is the set of positive zeros of the polynomial
$U_N^{(s)}(z)$ given\\
$\displaystyle U_N(z) = \sum_{j = 0}^{\br{\frac{N}{2}}} (-1)^j \frac{(N - j)!}{j!(N - 2j)!}(2z)^{N - 2j}$ are the Chebyshev Polynomials of the Second Kind and $U_N^{(s)}(z)$ is the $s$th derivative of $U_N(z)$. As an application of the results, formulas were obtained expressing the derivatives of Chebyshev polynomials of the second kind through linear combinations of Chebyshev polynomials of the second kind:
\begin{displaymath}
\frac{2^s}{s!}(1 - z^2)^sU_N^{(s)}(z) = (-1)^s \sum_{j = 0}^{s}(-1)^j{N-j \choose N-s} {N+s+1 \choose j} U_{N + s - 2j}(z).
\end{displaymath}
\end{abstract}

\section{Introduction}
The presented article is an announced continuation of the work\cite{ref1}, where the following polynomials were considered
$p(z) = 1 + \kappa(z + z^{N-1}) + z^N$ and $q(z) = 1 + \kappa(z - z^{N - 1}) - z^N$,
and for which, in particular, the following properties are established:
\begin{itemize}
\item all zeros of the quadrinomial $p(z)$ lie on the unit circle if and only if the inequalities are satisfied
$-1 \leq \kappa \leq \begin{cases}
1, & N \mbox{ is even} \\
\frac{N}{N -2}, & N \mbox{ is odd}
\end{cases}$;
\item for quadrinomial $q(z)$, the corresponding inequalities are
$-\frac{N}{N-2} \leq \kappa \leq \begin{cases}
1, & N \mbox{ is odd} \\
\frac{N}{N -2}, & N \mbox{ is even}
\end{cases}$;
\item let
$\cb{\nu_j}_{j = 1}^{\br{\frac{N - 3}{2}}}$ be the set of positive roots of the equation
$U'_{N - 2}(x) = 0$ and let $\gamma_j = 1 - 2(\nu_j)^2$;
also, let $\cb{\mu_j}_{j = 1}^{\br{\frac{N - 3}{2}}}$ be the set of positive roots of the equation $U_{N - 2}(x) = 0$ and let $\beta_j = 1 - 2(\mu_j)^2$ when $N$ is odd; then the following factorization formulas are valid for $q(z)$ (for $p(z)$ there are similar formulas):
\begin{enumerate}[\label=i)]
\item $\kappa = -\frac{N}{N - 2}$, $N$ odd
\begin{displaymath}
q(z) = (1 - z)^3\prod_{j = 1}^{\frac{N - 3}{2}}\br{1 + z^2 + 2z\gamma_j};
\end{displaymath}
\item $\kappa = -\frac{N}{N - 2}$, $N$ even
\begin{displaymath}
q(z) = (1 + z)(1 - z)^3\prod_{j = 1}^{\frac{N - 4}{2}}\br{1 + z^2 + 2z\gamma_j};
\end{displaymath}
\item $\kappa = \frac{N}{N - 2}$, $N$ even
\begin{displaymath}
q(z) = (1 - z)(1 + z)^3\prod_{j = 1}^{\frac{N - 4}{2}}\br{1 + z^2 + 2z\gamma_j};
\end{displaymath}
\item $\kappa = 1$, $N$ odd
\begin{displaymath}
q(z) = (1 - z)(1 + z)^2\prod_{j = 1}^{\frac{N - 3}{2}}\br{1 + z^2 + 2z\beta_j}.
\end{displaymath}
\end{enumerate}
\end{itemize}

Polynomial $P(z)$ with real coefficients is called reciprocal if $P(z) = \pm z^NP(\frac{1}{z})$ (additionally called symmetric or antisymmetric depending on the sign). The reciprocal polynomial has zeros located symmetrically with respect to the unit circle, i.e., if $z_0$ is a nontrivial root of a reciprocal polynomial, then $\frac{1}{z_0}$ is also a root of this polynomial. In particular, if $z_0$ lies on the unit circle then $\frac{1}{z_0}$ also lies on the unit circle.
\\

A significant number of works are devoted to polynomials whose zeros all lie on the unit circle--see, for example, \cite{ref3,ref4,ref5,ref6,ref7} where the importance of studying the properties of such polynomials is noted.
\\

The classical criterion of A. Cohn for the membership of all zeros of a polynomial on the unit circle is well known.

\begin{thm}\label{thm1}\cite{ref2}. All zeros of the polynomial $\displaystyle P(z) = \sum_{j = 0}^{N}a_j z^{N - j}$ with $a_0 \neq 0$ lie on the unit circle if and only if the polynomial $P(z)$ is reciprocal and all zeros of the polynomial $P'(z)$ belong to the closed central unit disc $\overline{\mathbb D} = \{z: \abs{z} \leq 1\}$.
\end{thm}

Notice that polynomials $p(z)$ and $q(z)$ are reciprocal ($p(z)$ is symmetric, $q(z)$ is antisymmetric) and all their zeros lie on the unit circle. In the extreme cases $\pa{\abs{\kappa} = \frac{N}{N - 2}, \ \abs{\kappa} = 1}$ these polynomials are related by simple relations with the Chebyshev polynomials of the second kind $U_N(x)$ and their first derivatives through their zeros\cite{ref1}.

In this paper, we consider a more general case of a reciprocal antisymmetric polynomial of the form $\displaystyle P(z) = \sum_{j = 0}^{s}(-1)^j\gamma_j\pa{z^j + z^{N + s + 1 - j}}$ with $(\gamma_0 = 1)$. The case of a symmetric polynomial can be considered similarly.

The purpose of this article is:
\begin{itemize}
\item obtaining the necessary conditions under which all zeros of the polynomial $P(z)$ lie on the unit circle, in the form of estimates for the coefficients $\gamma_j \ (j = 1,\ldots, s)$;
\item confirmation of the unimprovability of these estimates in the general case in the form of constructing extremal polynomials;
\item factorization of extremal polynomials through binomials with zeros on the unit circle, which are associated with derivatives of order $s \ (0 \leq s \leq N)$ of polynomials $U_N(x)$.
\end{itemize}

The obtained extremal properties of reciprocal polynomials are apparently new compared to the known ones (see, for example, \cite{ref8,ref9,ref10}).

As a consequence, we additionally obtain formulas representing the derivatives of Chebyshev polynomials of the second kind through linear combinations of the Chebyshev polynomials of the second kind. In \cite{ref11,ref12,ref13} it is noted that the exploration of such internal connections between Chebyshev polynomials of the second kind and their derivatives is an important and interesting task, and moreover, in the opinion of the authors, the results in this direction are clearly insufficient.

\section{Auxiliary results}

For every natural number $N \geq 0$ the family of Chebyshev polynomials of the second kind $U_N(x)$ are defined as follows (see, for example, \cite{ref11}):

\begin{equation} \label{cheby-eq}
U_N(z) = \sum_{j = 0}^{\br{\frac{N}2}}(-1)^j{N-j\choose j}(2z)^{N-2j},
\end{equation}
where we define
\begin{displaymath}
\br{\frac{N}{2}} = \begin{cases}
\frac{N}{2}, & \mbox{ if } N \mbox{ is even} \\
\frac{N - 1}{2}, & \mbox{ if } N \mbox{ is odd}
\end{cases}.
\end{displaymath}
Let us note that $\frac{(N - j)!}{j!(N - 2j)!} = {N - j \choose j}$. For $N < 0$ the family of Chebyshev polynomials of the second kind can be defined by the formula $U_{-N}(z) = -U_{N - 2}(z)$.

Formula~\eqref{cheby-eq} allows us to define polynomials $U_N(z)$ for all complex numbers $z$. We obtain explicit formulas for derivatives of order $s$ $\pa{0 \leq s \leq N}$ the Chebyshev polynomials of the second kind:
\begin{equation}\label{cheby-der-eqn1}
U_N^{(s)}(z) = s!\sum_{j = 0}^{\br{\frac{N - s}{2}}}(-1)^j {N-j\choose j}{N-2j\choose s}2^{N - 2j}z^{N - 2j - s}.
\end{equation}

Another explicit formula for the derivatives $U_N^{(s)}(z)$ was obtained on the basis of classical formulas of combinatorics in \cite{ref12}:
\begin{equation}\label{cheby-der-eqn2}
U_N^{(s)}(z) = \frac{2^s}{(s - 1)!}\sum_{j = 0}^{\br{\frac{N - s}{2}}}\frac{(N - j)!(j + s - 1)!}{j!(N - j - s + 1)!}(N - s - 2j + 1)U_{N - s - 2j}(z).
\end{equation}

Next, let's consider the polynomial $\displaystyle z^{\frac{N - s}{2}}U_N^{(s)}\pa{\frac{1}{2}\pa{z^{\frac{1}{2}} + z^{-\frac{1}{2}}}}$. In the case when $s = N$ notice that $\displaystyle z^{\frac{N - s}{2}}U_N^{(s)}\pa{\frac{1}{2}\pa{z^{\frac{1}{2}} + z^{-\frac{1}{2}}}} = 2^NN!$, as can be easily verified from \eqref{cheby-der-eqn1}. Thus, we turn our attention to the case when $s < N$.

\begin{lemma} \label{lem1}Let $\displaystyle \cb{\nu_j^{(s)}}_{j = 1}^{\br{\frac{N - s}{2}}}$ be the set of positive zeros of $U_N^{(s)}(z)$. Then, we have
\begin{eqnarray*}
\lefteqn{z^{\frac{N - s}{2}}U_N^{(s)}\pa{\frac{1}{2}\pa{z^{\frac{1}{2}} + z^{-\frac{1}{2}}}}} \\ &=& \frac{N!2^s}{(N - s)!}\begin{cases}
(1 + z)\prod_{j = 1}^{\frac{N - s - 1}{2}}\br{z^2 + 1 + 2z\pa{1 - 2\pa{\nu_j^{(s)}}^2}}, & N - s \mbox{ is odd} \\
\prod_{j = 1}^{\frac{N - s}{2}}\br{z^2 + 1 + 2z\pa{1 - 2\pa{\nu_j^{(s)}}^2}}, & N - s \mbox{ is even}
\end{cases}
\end{eqnarray*}
\end{lemma}

\begin{proof} Let $N - s$ be an even number. Since
$\displaystyle U_N^{(s)}(z) = \frac{N!2^N}{(N - s)!}z^{N - s} + \cdots$
we have that
$\displaystyle U_N^{(s)}(z) = \frac{N!}{(N - s)!}\prod_{j = 1}^{\frac{N - s}{2}}\br{z^2 - \pa{\nu_j^{(s)}}^2}$.
Substituting
$\displaystyle\frac{1}{2}\pa{z^{\frac{1}{2}} + z^{-\frac{1}{2}}}$
in place of $z$ we obtain the desired formula after multiplying by $z^{\frac{N - s}{2}}$. The case when the number $N - s$ is odd is shown similarly. The lemma is proved.
\end{proof}

From Lemma~\ref{lem1} it follows that the polynomial
$\displaystyle z^{\frac{N - s}{2}}U_N^{(s)}\pa{\frac{1}{2}\pa{z^{\frac{1}{2}} + z^{-\frac{1}{2}}}}$ has degree $N - s$, is reciprocal, and all its zeros lie on the circle. Indeed, $\pa{\nu_j^{(s)}}^2\le 1,$ therefore $|1 - 2\pa{\nu_j^{(s)}}^2|\le 1,$ while the polynomial $z^2+2bz+1$ has all its zeros on the circle if $|b|\le 1$.

\begin{lemma}\label{lem2}
The following formulas hold:
\begin{eqnarray*}
z^{\frac{N}{2}}U_N\pa{\frac{1}{2}\pa{z^{\frac{1}{2}} + z^{-\frac{1}{2}}}} &=& 1 + z + z^2 + \cdots + z^N \\
z^{\frac{N - 1}{2}}U_N'\pa{\frac{1}{2}\pa{z^{\frac{1}{2}} + z^{-\frac{1}{2}}}} &=& 2\sum_{j = 0}^{N - 1}(j + 1)(N - j)z^j.
\end{eqnarray*}
\end{lemma}

\begin{proof}
By definition, for any whole $N \geq 0$, we have
$\displaystyle U_N\pa{\cos\pa{\frac{t}{2}}} = \frac{\sin\pa{\frac{(N + 1)t}{2}}}{\sin\pa{\frac{t}{2}}}$ where
\begin{displaymath}
U_N\pa{\frac{1}{2}\pa{e^{i\frac{t}{2}} + e^{-i\frac{t}{2}}}} = \frac{e^{i\frac{(N + 1)t}{2}} - e^{-i\frac{(N + 1)t}{2}}}{e^{i\frac{t}{2}} + e^{-i\frac{t}{2}}} = \frac{e^{-i\frac{(N + 1)t}{2}}\pa{e^{i(N + 1)t} - 1}}{e^{-i\frac{t}{2}}\pa{e^{it} - 1}} = e^{-i\frac{Nt}{2}}\frac{1 - e^{i(N + 1)t}}{1 - e^{it}}.
\end{displaymath}
Thus, we have the following formula
\begin{displaymath}
U_N\pa{\frac{1}{2}\pa{z^{\frac{1}{2}} + z^{-\frac{1}{2}}}} = z^{-\frac{N}{2}}\frac{1 - z^{N + 1}}{1 - z} = z^{-\frac{N}{2}}\pa{1 + z + z^2 + \cdots + z^N}.
\end{displaymath}
holds for all $z \neq 1$ on the unit circle. Therefore, the two polynomials are equal for all $z \notin \{0,1\}$.

From this, we obtain the following
\begin{eqnarray*}
U_N'(x)\Big|_{x = \frac{1}{2}\pa{z^{\frac{1}{2}} + z^{-\frac{1}{2}}}} &=& \frac{d}{dz}\pa{z^{-\frac{N}{2}}\frac{1 - z^{N + 1}}{1 - z}} \frac{1}{\frac{d}{dz}\pa{\frac{1}{2}\pa{z^{\frac{1}{2}} + z^{-\frac{1}{2}}}}} \\
&=& 2z^{-\frac{N - 1}{2}} \frac{N(1 - z^{N + 2})-(N + 2)\pa{z - z^{N + 1}}}{(1 - z)^3},
\end{eqnarray*}
where
$U_N'\pa{\frac{1}{2}\pa{z^{\frac{1}{2}} + z^{-\frac{1}{2}}}} = 2z^{-\frac{N - 1}{2}}\sum_{j = 0}^{N - 1}(j + 1)(N - j)z^j$.
The lemma is proved.
\end{proof}

\section{Main Results}
\begin{thm}\label{thm2} The following formula holds.
\begin{equation}\label{cheby-der-eqn3}
z^{\frac{N-s}{2}}U_N^{(s)}\pa{\frac{1}{2}\pa{z^{\frac{1}{2}} + z^{-\frac{1}{2}}}} = {2^s}{s!}\sum_{j = 0}^{N - s}{N-j\choose s}{s+j\choose s}z^j.
\end{equation}
\end{thm}

\begin{proof}
For $s = 0,1$, formula~\eqref{cheby-der-eqn3} is true (Lemma~\ref{lem2}). We proceed by induction. Let us denote
\begin{displaymath}
A_{N,s,j} =(s!)^2{N-j\choose s}{s+j\choose s}= \frac{(N - j)!(j + s)!}{j!(N - j - s)!}.
\end{displaymath}
Hence,
\begin{equation}\label{A-rec-eqn}
A_{N,s + 1,j} = A_{N,s,j}(N - j - s)(j + s + 1) = A_{N,s,j + 1}(j + 1)(N - j).
\end{equation}
Further, let us note that
$\displaystyle \frac{d}{dz}\pa{\frac{1}{2}\pa{z^{\frac{1}{2}} + z^{-\frac{1}{2}}}} = \frac{1}{4}z^{-\frac{3}{2}}(z - 1)$. Let \eqref{cheby-der-eqn3} be satisfied for some $s$ $\pa{1 \leq s < N}$. Then, for $s + 1$:
\begin{eqnarray*}
U_N^{(s + 1)}\pa{\frac{1}{2}\pa{z^{\frac{1}{2}} + z^{-\frac{1}{2}}}} 
&=& \pa{\frac{d}{dz}\pa{U_N^{(s)}\pa{\frac{1}{2}\pa{z^{\frac{1}{2}} + z^{-\frac{1}{2}}}}}} \cdot \frac{1}{\frac{d}{dz}\br{\frac{1}{2}\pa{z^{\frac{1}{2}} + z^{-\frac{1}{2}}}}}\\
&=& \frac{2^s}{s!}\pa{\frac{d}{dz}\pa{\sum_{j = 0}^{N - s}A_{N,s,j}}z^{j - \frac{N - s}{2}}}\frac{4z^{\frac{3}{2}}}{z - 1} \\
&=& \frac{2^s}{s!}\frac{4z^{\frac{3}{2}}}{z - 1}\sum_{j = 0}^{N-s}A_{N,s,j}\pa{j - \frac{N-s}{2}}z^{j - \frac{N - s}{2} - 1} \\
&=& \frac{2^{s + 1}}{(s + 1)!}\frac{z^{-\frac{N-s-1}{2}}}{z - 1}\sum_{j = 0}^{N-s}A_{N,s,j}(2j - N + s)(s + 1)z^j.
\end{eqnarray*}
It is necessary to check the identity
\begin{displaymath}
(z - 1)\sum_{j = 0}^{N - s - 1}A_{N,s + 1,j}z^j \equiv \sum_{j = 0}^{N - s}A_{N,s,j}(2j - N + s)(s + 1)z^j.
\end{displaymath}
To do this, we compare coefficients on powers of $z$. Using \eqref{A-rec-eqn} yields the following for powers of $z$:
\begin{itemize}
\item $z^0$: $-A_{N,s + 1,0} \equiv A_{N,s,0}(-N + s)(s + 1)$;
\item $z^{N - s}$: $A_{N,s + 1,N - s - 1} \equiv A_{N,s,N - s}(N - s)(s + 1)$;
\item $z^j$ $\pa{j = 1, \ldots, N - s - 1}$:
\begin{eqnarray*}
A_{N,s + 1,j - 1} - A_{N,s + 1,j} &=& A_{N,s,j}(j(N - j + 1) - (j + s + 1)(N - j - s)) \\
&\equiv& A_{N,s,j}(2j - N + s)(s + 1).
\end{eqnarray*}
\end{itemize}
By Strong Induction, the theorem is proved.
\end{proof}

Substituting $z = e^{2it}$ into formula~\eqref{cheby-der-eqn3} and letting $x = \cos(t)$, separating the real and imaginary parts yields
\begin{displaymath}
U_N^{(s)}(x) = {2^s}{s!}\sum_{j = 0}^{N - s} 
{N-j\choose s} {s+j\choose s}T_{N - s - 2j}(x),
\quad
 \sum_{j = 0}^{N - s} 
{N-j\choose s} {s+j\choose s} U_{N - s - 2j - 1}(x) = 0,
\end{displaymath}
where $T_N(x)$ are the Chebyshev polynomials of the first kind. 

Since $T_{N + 1}(x) = xU_N(x) - U_{N - 1}(x) = \frac{1}{2}\pa{U_{N + 1}(x) - U_{N - 1}(x)}$, the following relations are valid:
\begin{displaymath}
U_N^{(s)}(x) = {2^s}{s!}\sum_{j = 0}^{N - s} {N-j\choose s}{s+j\choose s}
U_{N - s - 2j}(x),
\end{displaymath}
\begin{displaymath}
U_N^{(s)}(x) = -{2^s}{s!}\sum_{j = 0}^{N - s} {N-j\choose s}{s+j\choose s}
U_{N - s - 2j-2}(x).
\end{displaymath}
Now, taking into account that the polynomial
$\displaystyle \sum_{j = 0}^{N - s} {N-j\choose s}{s+j\choose s}z^j$
is reciprocal, then we obtain the formula
\begin{displaymath}
U_N^{(s)}(x) = \Theta_{N,s} + {2^s}{s!}\sum_{j = 0}^{[\frac{N - s}2]} {N-j\choose s}{s+j\choose s} \pa{U_{N - s - 2j}(x) - U_{N - s - 2j - 2}(x)},
\end{displaymath}
where we define
\begin{displaymath}
\Theta_{N,s} = \begin{cases}
0, & N - s\mbox{ is odd} \\
-{2^s}{s!}\pa{\frac{N+s}2\choose\frac{N-s}2}^2, & N - s\mbox{ is even}
\end{cases}.
\end{displaymath}
All obtained formulas are equivalent to \eqref{cheby-der-eqn2}.

We normalize the polynomial $\displaystyle z^{\frac{N-s}{2}}U_N^{(s)}\pa{\frac{1}{2}\pa{z^{\frac{1}{2}} + z^{-\frac{1}{2}}}}$ so that the free term is equal to one, i.e. we define a polynomial $\displaystyle F_{N,s}(z) = 2^{-s}\frac{(N - s)!}{N!}z^{\frac{N-s}{2}}U_N^{(s)}\pa{\frac{1}{2}\pa{z^{\frac{1}{2}} + z^{-\frac{1}{2}}}}$ with $F_{N,s}(0) = 1$. Let us note that $F_{N,N}(z) \equiv 1$.

It follows from Lemma~\ref{lem1} that
\begin{displaymath}
F_{N,s}(z) = \left\{\begin{array}{ll}
(1 + z), & N - s \mbox{ is odd} \\
1, & N - s \mbox{ is even}
\end{array}\right\} \cdot
\prod_{j = 1}^{\br{\frac{N - s}{2}}}\br{z^2 + 1 + 2z\pa{1 - 2\pa{\nu_j^{(s)}}^2}}.
\end{displaymath}
Some examples are listed below.
\begin{enumerate}
\item For $s = 0$, by Lemma~\ref{lem2} we have
\begin{displaymath}
F_{N,0}(z) = \frac{1 - z^{N + 1}}{1 - z};
\end{displaymath}
hence,
\begin{displaymath}
1 - z^{N + 1} = (1 - z)\left\{\begin{array}{ll}
(1 + z), & N \mbox{ is odd} \\
1, & N \mbox{ is even}
\end{array}\right\} \cdot \br{z^2 + 1 + 2z(1 - 2\pa{\nu_j^{(0)}}^2}
\end{displaymath}
where $\cb{\nu_j^{(0)}}_{j = 1}^{\br{\frac{N}{2}}}$ is the set of positive zeros of polynomial $U_N(z)$.

\item For $s = 1$, by Lemma~\ref{lem2} and manipulation of geometric series we have
\begin{displaymath}
F_{N,1}(z) = \frac{1}{(1 - z)^3}\pa{1 - z^{N + 2} - \frac{N + 2}{N} (z - z^{N + 1})};
\end{displaymath}
hence,
\begin{eqnarray*}
\lefteqn{1 - z^{N + 2} - \frac{N + 2}{N}(z - z^{N + 1})}  \\
&=& (1 - z)^3\left\{\begin{array}{ll}
(1 + z), & N \mbox{ is even} \\
1, & N \mbox{ is odd}
\end{array}\right\} \cdot \br{z^2 + 1 + 2z(1 - 2\pa{\nu_j^{(1)}}^2}
\end{eqnarray*}
where $\cb{\nu_j^{(1)}}_{j = 1}^{\br{\frac{N - 1}{2}}}$ is the set of positive zeros for $U_N'(z)$

\item For $s = 2$, we have
\begin{displaymath}
F_{N,2}(z) = \frac{1}{(1 - z)^5}\pa{1 - z^{N + 3} - \frac{2(N + 3)}{N}(z - z^{N + 2}) + \frac{(N + 2)(N + 3)}{N(N - 1)}(z^2 - z^{N + 1})};
\end{displaymath}
\begin{eqnarray*}
\lefteqn{1 - z^{N + 3} - \frac{2(N + 3)}{N}(z - z^{N + 2}) + \frac{(N + 2)(N + 3)}{N(N - 1)}(z^2 - z^{N + 1})} \\
&=& ( 1- z)^5\left\{\begin{array}{ll}
(1 + z), & N \mbox{ is odd} \\
1, & N \mbox{ is even}
\end{array}\right\} \cdot
\prod_{j = 1}^{\br{\frac{N - 2}{2}}}\br{z^2 + 1 + 2z(1 - 2\pa{\nu_j^{(2)}}^2},
\end{eqnarray*}
where $\cb{\nu_j^{(2)}}_{j = 1}^{\br{\frac{N - 2}{2}}}$ is the set of positive zeros of $U_N''(z)$

\item For $s = 3$, we have
\begin{eqnarray*}
F_{N,3}(z) &=& \frac{1}{(1 - z)^7}\left[1 - z^{N + 4} - \frac{3(N + 4)}{N}(z - z^{N + 3}) + \frac{3(N + 3)(N + 4)}{(N - 1)N}(z^2 - z^{N + 2})\right. \\
& & \qquad \qquad - \left.\frac{(N + 2)(N + 3)(N + 4)}{(N -  2)(N - 1)N}(z^3 - z^{N + 1}) \right]
\end{eqnarray*}
\begin{eqnarray*}
\lefteqn{1 - z^{N + 4} - \frac{3(N + 4)}{N}(z - z^{N + 3}) + \frac{3(N + 3)(N + 4)}{(N - 1)N}(z^2 - z^{N + 2})} \\
\lefteqn{\qquad \qquad - \frac{(N + 2)(N + 3)(N + 4)}{(N -  2)(N - 1)N}(z^3 - z^{N + 1})} \\
&=& (1 - z)^7\left\{\begin{array}{ll}
(1 + z), & N \mbox{ is even} \\
1, & N \mbox{ is odd}
\end{array}\right\}
\prod_{j = 1}^{\br{\frac{N - 3}{2}}}\br{z^2 + 1 + 2z(1 - 2\pa{\nu_j^{(3)}}^2}
\end{eqnarray*}
where $\cb{\nu_j^{(3)}}_{j = 1}^{\br{\frac{N - 3}{2}}}$ is the set of positive zeros of $U_N^{(3)}(z)$

\item For $s = N - 2$, we have $F_{N,N - 2}(z) = (1 + z)^2$.

\item For $s = N - 1$, we have $F_{N,N - 1}(z) = 1 + z$.
\end{enumerate}

The constructed examples suggest the following representation of the polynomial $F_{N,s}(z)$.

\begin{thm}\label{thm3}
The following formula holds:
\begin{equation}\label{eqn6}
(1 - z)^{2s + 1}F_{N,s}(z) = {N\choose s}^{-1} \sum_{j = 0}^{s}(-1)^j {N+s+1\choose j}{N-j\choose N-s}\pa{z^j - z^{N + s + 1 - j}}.
\end{equation}
\end{thm}

\begin{proof} Formula \eqref{eqn6} is equivalent to
\begin{equation}\notag
(1 - z)^{2s + 1}F_{N,s}(z) = \frac{s!(N + s + 1)!}{N!}\sum_{j = 0}^{s}(-1)^j \frac{(N - j)!}{j!(s - j)!(N + s + 1- j)!}\pa{z^j - z^{N + s + 1 - j}}.
\end{equation}
By virtue of Theorem~\ref{thm1}, it is required to prove that
\begin{eqnarray}\label{eq7}
\lefteqn{(1 - z)^{2s + 1}\sum_{j = 0}^{N - s}\frac{(N - j)!(j + s)!}{j!(N - j - s)!}z^j} \nonumber \\&=& \frac{s!s!(N + s + 1)!}{(N - s)!}\sum_{j = 0}^{s}(-1)^j \frac{(N - j)!}{j!(s - j)!(N + s + 1- j)!}\pa{z^j - z^{N + s + 1 - j}}.
\end{eqnarray}

Let's consider the polynomial on the left side of equation~\eqref{eq7} and find $A_p$, where $A_p$ is the coefficient on $z^p$. We use the binomial formula and the formula for the convolution of the coefficients of the product of two polynomials. We obtain
\begin{displaymath}
A_p = \sum_{j = 0}^{N - s}(-1)^{p - j}\frac{(N - j)!(j + s)!(2s + 1)!}{j!(N - j - s)!(p - j)!(2s + 1 + j - p)!}.
\end{displaymath}
Note that
\begin{displaymath}
A_p =(-1)^p\frac{N!s!(2s + 1)!}{p!(N - s)!(2s + 1- p)!} \cdot \prescript{3}{F}{2}\pa{\br{-p,s + 1,-N + s};\br{-N,2s + 2 - p};1},
\end{displaymath}
 where
$\displaystyle\prescript{3}{F}{2}\pa{\br{a_1,a_2,a_3};\br{b_1,b_2};z}$
is a generalized hypergeometric function in \cite{ref14} with $a_1 = -p$, $a_2 = s + 1$, $a_3 = -N + s$, $b_1 = -N$, $b_2 = 2s + 2 - p$, $z = 1$. Further using the formula \cite[5.2.4,(1)]{ref14}, we obtain  the following:
$\displaystyle\prescript{3}{F}{2}\pa{\br{-n,a,b};\br{c,1 + a + b - c - n}; 1} = \frac{(c - a,n)(c - b,n)}{(c,n)(c - a - b,n)}$, $n = 0,1, \ldots$.
Here $(x,k)$ is the pochhammer symbol, 
$$(x,k) = \lim_{\tau \rightarrow 0} \frac{\Gamma(1 - x + \tau)}{\Gamma(1 - x - k + \tau)}=
\lim_{\tau \rightarrow 0} \frac{\Gamma(x +  k + \tau)}{\Gamma(x + \tau)},$$ 
and $\Gamma(\cdot)$ is the Gamma function. Then,
\begin{displaymath}
A_p = (-1)^p \frac{s!s!(N + s + 1)!(N - p)!}{p!(N - s)!(s - p)!(N + s + 1 - p)!}.
\end{displaymath}
But this expression exactly coincides with the coefficient of $z^p$ on the right-hand side of equality~\eqref{eq7}. The theorem is proven.
\end{proof}

{\it Remark}. $A_p = 0$ when $s < p < N + 1$. Consequently, we have the following 
\begin{corollary}\label{cor:1}
\begin{eqnarray*}
\lefteqn{\sum_{j = 0}^{s}(-1)^j{s \choose j}{N + s + 1 \choose j}\pa{{N \choose j}}^{-1}\pa{z^j - z^{N + s + 1 - j}}} \\
&=& (1 - z)^{2s + 1} \prod_{j = 1}^{\br{\frac{N - s}{2}}}\br{z^2 + 1 + 2z(1 - 2\pa{\nu_j}^2)}\left\{\begin{array}{ll}
(1 + z), & N - s \mbox{ is odd} \\
1, & N - s \mbox{ is even}
\end{array}\right\},
\end{eqnarray*}
where $\cb{\nu_j}_{j = 1}^{\br{\frac{N - s}{2}}}$ is the set of positive roots of $U_N^{(s)}(z)$.
\end{corollary}

Note that 
$$
{s\choose j}\left({N\choose j}\right)^{-1}=\left({N\choose s}\right)^{-1}{N-j\choose N-s}.
$$

\section{Extremal properties of polynomials with zeros on the unit circle}

Let us consider an antisymmetric reciprocal polynomial of the form
\begin{equation}\label{eqn8}
P(z) = \sum_{j = 0}^s(-1)^j\gamma_j\pa{z^j - z^{N + s + 1 - j}},
\end{equation}
where $\gamma_0 = 1$ for $0 \leq s \leq N$. Obviously, $P(1) = 0$, $P(-1) = 0$, where $N - s$ is odd.

\begin{lemma}\label{lem3}
If the inequality $\displaystyle \sum_{j = 1}^{s} \abs{\gamma_j} \leq 1$ is satisfied, then all of the zeros of the polynomial $P(z)$ are located on the unit circle.
\end{lemma}

\begin{proof}
Let's define the derivative
\begin{displaymath}
P'(z) = -(N + s + 1)z^{N + s} + \sum_{j = 1}^{s}(-1)^j\gamma_j\pa{jz^{j - 1} - (N + s + 1 - j)z^{N + s - j}}
\end{displaymath}
If the sum of the absolute values of the coefficients of the polynomial $\displaystyle \frac{-1}{N + s + 1}P'(z)$ (excluding the coefficient for the highest power of $z$) does not exceed 1, Rouch\'{e}'s Theorem tells us that all zeros of this polynomial belong to the closed central unit disc $\overline{\mathbb D} = \cb{z \ : \ \abs{z} \leq 1}$.  Therefore, by Theorem~\ref{thm1}, all of the zeros of the polynomial $P(z)$ are located on the unit circle. The lemma is proved.
\end{proof}

\begin{thm}\label{thm4}
If all zeros of a polynomial $P(z)$ are located on the unit circle, then
\begin{equation}\label{eqn9}
\abs{\gamma_j} \leq {s \choose j}{N + s + 1 \choose j}\pa{{N \choose j}}^{-1}, \ j = 1,\ldots,s.
\end{equation}
Moreover, these estimates cannot be improved in general.
\end{thm}

\begin{proof}
Let us define the derivative of order $s + 1$
\begin{eqnarray*}
P^{(s + 1)}(z) &=& \pa{-\sum_{j = 0}^{s}(-1)^j\gamma_jz^{N + s + 1 - j}}^{(s + 1)} \\
&=& -\frac{(N + s+ 1)!}{N!}z^{N - s}\sum_{j = 0}^{s}(-1)^j\gamma_j\frac{N!}{(N + s + 1)!}\frac{(N + s + 1 - j)!}{(N - j)!}z^{s - j}.
\end{eqnarray*}
If all zeros of a polynomial $P(z)$ are located on the unit circle, then all of the zeros of the polynomial $P'(z)$ lie inside the disc $\overline{\mathbb D}$. Since the zeros of the derivative of a polynomial lie inside the closed convex hull of the zeros of this polynomial, then all the zeros of the polynomial $P^{(s + 1)}(z)$ must lie inside the disc $\overline{\mathbb D}$. From Vieta's theorem it follows that the coefficients of the polynomial $\displaystyle \sum_{j = 0}^{s}(-1)^j\gamma_j\frac{N!}{(N + s + 1)!}\frac{(N + s + 1 - j)!}{(N - j)!}z^{s - j}$ must satisfy the constraints:
$\displaystyle |\gamma_j|\frac{N!}{(N + s + 1)!}\frac{(N + s + 1 - j)!}{(N - j)!} \leq {s \choose s - j}$ for $j = 1, \ldots, s$, which are equivalent to conditions \eqref{eqn9}. Generally speaking, it is impossible to reduce the validity of even one of the inequalities \eqref{eqn9}, as shown by the example from the Corollary \ref{cor:1}. 
The theorem is proved.
\end{proof}

Thus, the set in the space of coefficients of the polynomial \eqref{eqn8}, for which all its zeros lie on the unit circle, is contained in the parallelepiped
\begin{displaymath}
\hat{\Pi} = \cb{\gamma_1,\ldots,\gamma_s \ : \ \abs{\gamma_j} \leq {s \choose j}{N + s + 1 \choose j}\pa{{N \choose j}}^{-1}, \ j = 1,\ldots,s }
\end{displaymath}
and contains a parallelepiped (cube)
\begin{displaymath}
\tilde{\Pi} = \cb{\gamma_1,\ldots,\gamma_s \ : \ \sum_{j = 1}^{s}\abs{\gamma_j} \leq 1}.
\end{displaymath}
Moreover, without additional assumptions on the coefficients, neither the sides of the parallelepiped $\hat{\Pi}$ nor the edges of the cube $\tilde{\Pi}$ can be reduced or increased. For the latter case, an example would be a polynomial $P(z) = \pa{1 - z^{N + 3}} + \gamma_2\pa{z^2 - z^{N + 1}}$. Then, $\gamma_2 \leq 1$ for odd $N$ and $\gamma_2 \geq -1$ for even $N$ \cite{ref1}.

\section{Some examples}

We will find $\displaystyle P(e^{it}) = -2ie^{i\frac{N + s + 1}{2}t} S(t)$ where $\displaystyle S(t) = \sum_{j = 0}^{s}(-1)^j\gamma_j\sin\pa{\pa{\frac{N + s + 1}{2} - j}t}$.

From this formula we immediately obtain a criterion for all zeros of polynomial \eqref{eqn8} to belong to the unit circle: in order for all zeros of polynomial \eqref{eqn8} to lie on the unit circle, it is necessary and sufficient that the trigonometric polynomial $S(t)$ would have $\displaystyle \br{\frac{N + s + 3}{2}}$ zeros on $\br{0,\pi}$ taking into account their multiplicity. Henceforth, we will refer to this as ``condition A". 

However, this criterion can only be used to check whether all zeros of the polynomial \eqref{eqn8} belong to the unit circle when all coefficients of the polynomial are exactly known. To construct sets in the space of coefficients for which all zeros of the polynomial \eqref{eqn8} belong to the unit circle, condition A must be supplemented with other conditions, which can be obtained, for example, from Theorem~\ref{thm1}.

We will find $\displaystyle P'(e^{it}) = -i(N + s + 1)e^{i\frac{N + s - 1}{2}t}\br{iS(t) + C(t)}$, where
\begin{eqnarray*}
C(t) &=& \frac{2}{N + s + 1}S'(t) \\
&=& \sum_{j = 0}^{s}(-1)^j\gamma_j\pa{\frac{N + s + 1 - 2j}{N + s + 1}}\cos\pa{\pa{\frac{N + s + 1}{2} - j}t}. 
\end{eqnarray*}

Hence, if $P'(z_0)  = 0$ with $\abs{z_0} = 1$ then $P(z_0) = 0$. According to Theorem~\ref{thm1}, the boundary of a set in the space of coefficients $\gamma_1, \ldots, \gamma_s$ for which all zeros of the polynomial (8) belong to the unit circle, is determined from the following sufficient condition: there exists $z_0$ such that $P'(z_0) = 0$ with $\abs{z_0} = 1$, and all the other zeros of the polynomial $P'(z)$ lie in the disc $\overline{\mathbb D}$. Moreover,
\begin{displaymath}
P'(1) = i\pa{-(N + s + 1) - \sum_{j = 1}^{s}(-1)^j(N + s + 1 - 2j)\gamma_j},
\end{displaymath}
\begin{displaymath}
P'(-1) = \begin{cases}
\displaystyle -(N + s + 1) - \sum_{j = 1}^{s}(N + s + 1 - 2j)\gamma_j, & N - s \mbox{ is odd} \\
\displaystyle i(N + s + 1)\pa{1 + \sum_{j = 1}^{s}\gamma_j}, & N - s \mbox{ is even}
\end{cases}.
\end{displaymath}
For every $\tau \in \pa{0,\pi}$ define a set
\begin{displaymath}
U_\tau = \cb{\gamma_1, \ldots, \gamma_s \ : \ S(\tau) = 0, C(\tau) = 0,\; \mbox{Condition A is fulfilled}}.
\end{displaymath}

Let
\begin{displaymath}
U_0 = \cb{\gamma_1,\ldots,\gamma_s \ : \ -(N + s + 1) - \sum_{j = 1}^{s}(-1)^j(N + s + 1 - 2j)\gamma_j = 0,\; \mbox{Condition A is fulfilled}},
\end{displaymath}
\begin{displaymath}
U_\pi = \cb{\gamma_1,\ldots,\gamma_s \ : \ -(N + s + 1) - \sum_{j = 1}^{s}(N + s + 1 - 2j)\gamma_j = 0, \mbox{Condition A is fulfilled}},
\end{displaymath}
for odd $N - s$, and
\begin{displaymath}
U_\pi = \cb{\gamma_1,\ldots,\gamma_s \ : \ 1 + \sum_{j = 1}^{s}\gamma_j = 0,\; \mbox{Condition A is fulfilled}}
\end{displaymath}
for even $N - s$.

Then, many $\displaystyle \bigcup_{\tau \in \br{0,\pi}} \cb{U_\tau}$ form a boundary for a set in the space of coefficients $\gamma_1, \ldots, \gamma_s$ in which all zeros of polynomial \eqref{eqn8} belong to the unit circle.

For example, see below cases of when $s = 1$ and $s = 2$.
\begin{itemize}
\item $s = 1$\cite{ref1}:
\begin{displaymath}
S(t) = \sin\pa{\frac{N + 2}{2}t} - \gamma_1\sin\pa{\frac{N}{2}t}, \ C(t) = \cos\pa{\frac{N + 2}{2}t} - \frac{N}{N + 2}\gamma_1\cos\pa{\frac{N}{2}t}.
\end{displaymath}

Excluding $\gamma_1$ 
\begin{displaymath}
\begin{cases}
S(t) = 0 \\
C(t) = 0
\end{cases},
\end{displaymath}
yields $(N + 1)\sin(t) - \sin\pa{(N + 1)t} = 0$. This last equation has no solutions in the interval $\pa{0,\pi}$. Consequently, $U_\tau = \emptyset$ for $\tau \in \pa{0,\pi}$. Clearly, 
\begin{displaymath}
U_0 = \cb{\gamma_1 \ : \ \gamma_1 = \frac{N + 2}{N}}, \ U_{\pi} = \cb{\gamma_1 \ : \ \gamma_1 = -\frac{N + 2}{N}}
\end{displaymath}
for odd $N - 1$, and $\displaystyle U_{\pi} = \cb{\gamma_1 \ : \ \gamma_1 = -1}$ for even $N - 1$. Thus, a set in the coefficient space of the polynomials \eqref{eqn8} for which all zeros of the polynomial belong to the unit circle is defined by many $\displaystyle \cb{\gamma_1 \ : \ \abs{\gamma_1} \leq \frac{N + 2}{N}}$ for even $N$ and many $\displaystyle \cb{\gamma_1 \ : \ -1 \leq \gamma_1 \leq \frac{N + 2}{N}}$ for odd $N$.

\item $s = 2$:

Consider the system
\begin{displaymath}
\begin{cases}
\displaystyle \sin\pa{\frac{N + 3}{2}t} - \gamma_1\sin\pa{\frac{N + 1}{2}t} + \gamma_2\sin\pa{\frac{N - 1}{2}t} = 0 \\
\displaystyle (N + 3)\cos\pa{\frac{N + 3}{2}t} - (N + 1)\gamma_1\cos\pa{\frac{N + 1}{2}t} + (N - 1)\gamma_2\cos\pa{\frac{N - 1}{2}t} = 0
\end{cases}.
\end{displaymath}
The solution to this system can be represented as
\begin{displaymath}
\gamma_1 = \frac{(N + 1)\sin(2t) - 2\sin\pa{(N + 1)t}}{N\sin(t) - \sin(Nt)}, \ \gamma_2 = \frac{(N + 2)\sin(t) - \sin\pa{(N + 2)t}}{N\sin(t) - \sin(Nt)}.
\end{displaymath}
Then,
\begin{eqnarray*}
\bigcup_{\tau \in \pa{0,\pi}}\cb{U_\tau} &=& \left\{ \gamma_1, \gamma_2 \ : \ \gamma_1 = \frac{(N + 1)\sin(2t) - 2\sin\pa{(N + 1)t}}{N\sin(t) - \sin(Nt)}, \right. \\
& & \qquad \qquad \left. \gamma_2 = \frac{(N + 2)\sin(t) - \sin\pa{(N + 2)t}}{N\sin(t) - \sin(Nt)},\; t \in \pa{0,\pi}\right\}.
\end{eqnarray*}
We find,
\begin{displaymath}
U_0 = \cb{\gamma_1,\gamma_2 \ : \ \gamma_2 = \frac{N + 1}{N - 1}\gamma_1 - \frac{N + 3}{N - 1}, \ \gamma_1 \in \br{0,2\frac{N + 3}{N}}},
\end{displaymath}
\begin{displaymath}
U_{\pi} = \cb{\gamma_1,\gamma_2 \ : \ \gamma_2 = -\frac{N + 1}{N - 1}\gamma_1 - \frac{N + 3}{N - 1}, \ \gamma_1 \in \br{-2\frac{N + 3}{N},0}},
\end{displaymath}
when $N$ is odd, and
\begin{displaymath}
U_{\pi} = \cb{\gamma_1,\gamma_2 \ : \ \gamma_2 = -\gamma_1 - 1, \ \gamma_1 \in \br{-2\frac{N + 1}{N},\frac{2}{N}}},
\end{displaymath}
when $N$ is even.
\end{itemize}

Let us depict the boundary of the set in the plane of coefficients $\gamma_1, \ \gamma_2$ for \eqref{eqn8} (for $s = 2$), in which all the zeros of this polynomial belong to the unit circle. In Fig.~\ref{fig1}, we plot the rectangle
$$
\hat\Pi:=\left\{
 \gamma_1, \gamma_2: |\gamma_1|\le 2\frac{N+3}N,\; |\gamma_2|\le \frac{(N+3)(N+2)}{N(N-1)} 
 \right\}
$$
and the square $\tilde \Pi:=\{\gamma_1, \gamma_2: |\gamma_1|+|\gamma_2|\le 1\}$ when $s = 2$ for $N = 11$ and $N = 12$. The solid line boundaries between the square and rectangle are formed by $U_0$, $U_{\pi}$ and $\displaystyle \bigcup_{\tau \in (0,\pi)} \{U_{\tau}\}$.

\begin{figure}[h]
\begin{center}
\includegraphics[width=0.45\linewidth]{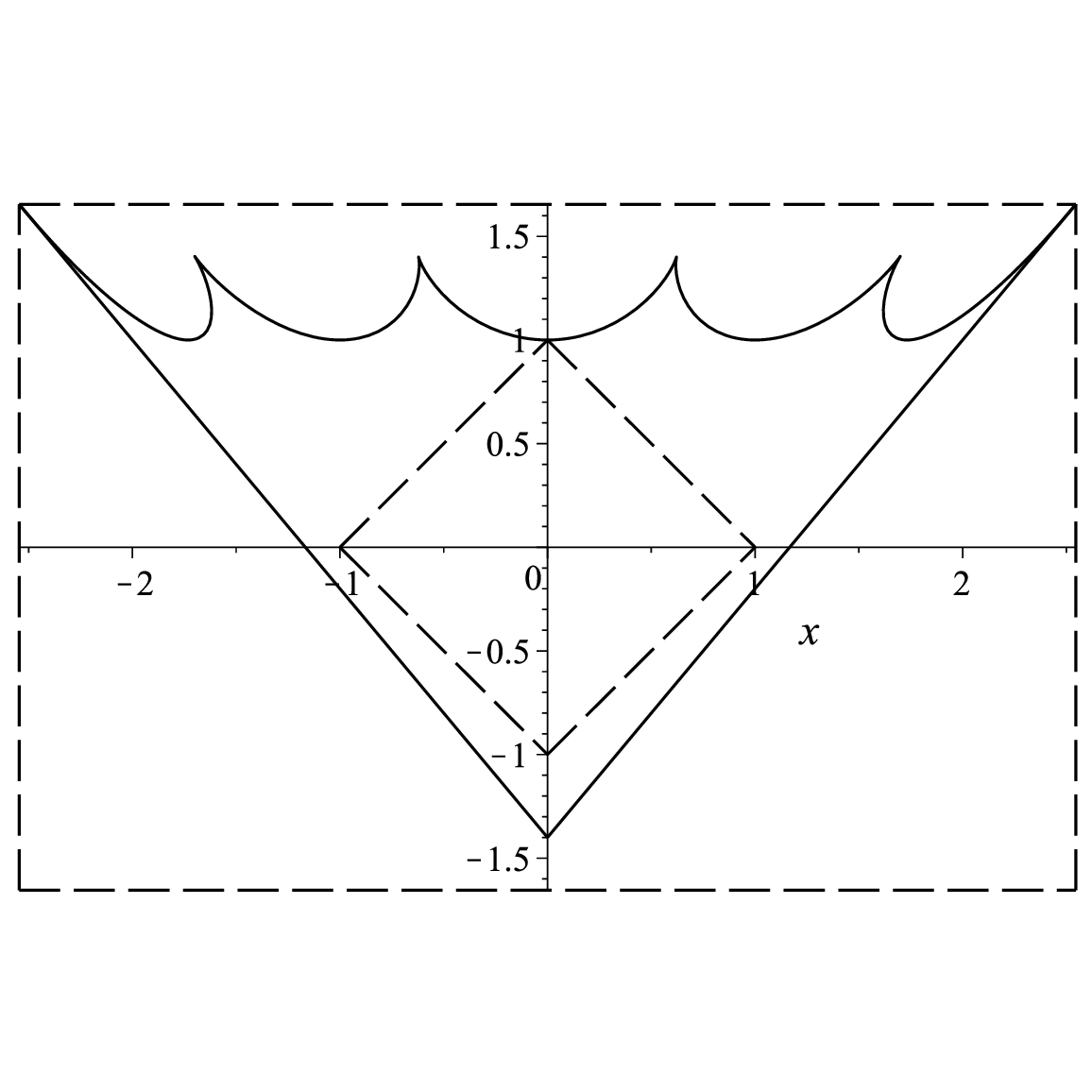}
\includegraphics[width=0.45\linewidth]{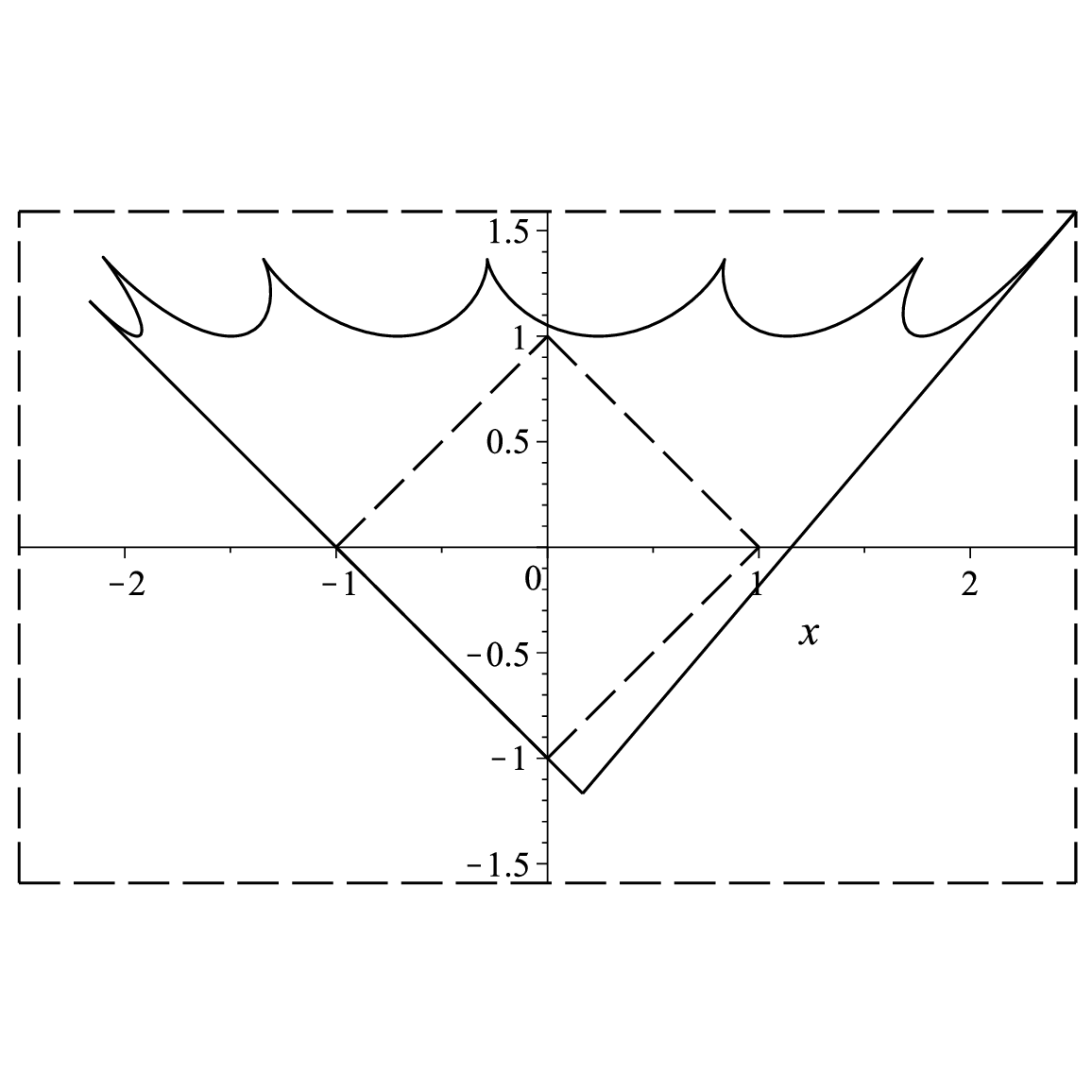}
\centerline{(a)\hspace{7cm} (b)}
\end{center}
\caption{Sets in the plane of coefficients $\gamma_1, \ \gamma_2$ for polynomial \eqref{eqn8}, in which all zeros of this polynomial belong to the unit circle for $N = 11$ (a) and $N = 12$ (b). The dashed lines shows the boundaries of the parallelepiped $\hat{\Pi}$ and cube $\tilde{\Pi}$ while the solid lines show the boundaries formed by $U_0$, $U_{\pi}$, and $\displaystyle \bigcup_{\tau \in (0,\pi)} \{U_{\tau}\}$.}\label{fig1}
\end{figure}

Note that when $N - s$ is even we have $\displaystyle P(-z) = \sum_{j = 0}^{s}\gamma_j\pa{z^j + z^{N + s + 1 - j}}$. All coefficients of an extremal polynomial are positive. The importance of studying the properties of polynomials with positive coefficients and with zeros on the unit circle was emphasized, for example, in \cite{ref3,ref15}.

\section{Representation of derivatives of the Chebyshev polynomial of the second kind in terms of Chebyshev polynomials of the second kind}

Formula \eqref{eqn6} leads to a new formula for the relationship between Chebyshev polynomials and their derivatives.

\begin{thm}\label{thm5}
The representation below is valid:
\begin{eqnarray} \label{eqn10}
U_N^{(s)}(z) &=& (-1)^s2^{-s}\frac{s!(N + s + 1)!}{(N - s)!}\frac{1}{(1 - z^2)^s} \nonumber \\
& & \qquad \cdot \sum_{j = 0}^{s}(-1)^j \frac{(N - j)!}{j!(s - j)!(N + s + 1 - j)!}U_{N + s - 2j}(z).
\end{eqnarray}
\end{thm}

\begin{proof}
We write formula \eqref{eqn6} in the form
\begin{eqnarray*}
U_{N}^{(s)}\pa{\frac{1}{2}\pa{z^{\frac{1}{2}} + z^{-\frac{1}{2}}}} &=& 2^{s}\frac{s!(N + s + 1)!}{(N - s)!} \frac{1}{(1 - z)^{2s + 1}} \\
& & \qquad \cdot \sum_{j = 0}^{s} (-1)^j \frac{(N - j)!}{j!(s - j)!(N + s + 1 - j)!}\pa{z^{j - \frac{N}{2} + \frac{s}{2}} - z^{\frac{N}{2} + \frac{3s}{2} + 1 - j}}.
\end{eqnarray*}
Notice that
$\displaystyle (1 - z)^{2s + 1} = z^{s + \frac{1}{2}}\pa{z^{-\frac{1}{2}} - z^{\frac{1}{2}}}^{2s + 1}$. Then
\begin{eqnarray*}
U_{N}^{(s)}\pa{\frac{1}{2}\pa{z^{\frac{1}{2}} + z^{-\frac{1}{2}}}} &=& 2^{s}\frac{s!(N + s + 1)!}{(N - s)!} \frac{1}{ \pa{z^{-\frac{1}{2}} - z^{\frac{1}{2}}}^{2s}} \\
& & \qquad \cdot \sum_{j = 0}^{s} (-1)^j \frac{(N - j)!}{j!(s - j)!(N + s + 1 - j)!} \\
& & \qquad \cdot \pa{\frac{z^{j - \frac{N}{2} - \frac{s}{2} - \frac{1}{2}} - z^{\frac{N}{2} + \frac{s}{2} + \frac{1}{2} - j}}{z^{-\frac{1}{2}} - z^{\frac{1}{2}}}}.
\end{eqnarray*}
Make the substitution $z = e^{it}$ and notice that 
\begin{displaymath}
\pa{z^{-\frac{1}{2}} - z^{\frac{1}{2}}}^{2s} = (2i)^{2s}\pa{\sin^2\pa{\frac{t}{2}}}^s = (-1)^s\cdot 2^s\pa{1 - \cos^2\pa{\frac{t}{2}}}^s.
\end{displaymath}
We obtain \eqref{eqn10} when $\displaystyle z = \cos\pa{\frac{t}{2}}$. But if two polynomials coincide on a segment, then they coincide for all complex arguments. The theorem is proven.
\end{proof}

Formula \eqref{eqn10} can be written as
\begin{displaymath}
\frac{2^s}{s!}(1 - z^2)^sU_{N}^{(s)}(z) = (-1)^s \sum_{j=0}^s (-1)^j {N-j \choose N-s}{N + s + 1 \choose j}U_{N + s - 2j}(z).
\end{displaymath}
Obviously, this formula is different from formulas \eqref{cheby-der-eqn1} and \eqref{cheby-der-eqn2}.

We demonstrate for specific values of $s$ below.
\begin{enumerate}
\item $s = 1$,
\begin{displaymath}
U_N'(z) = \frac{1}{2}\frac{1}{1 - z^2}\pa{(N + 2)U_{N - 1}(z) - NU_{N + 1}(z)};
\end{displaymath}

\item $s = 2$,
\begin{eqnarray*}
U_N''(z) &=& \frac{1}{4}\frac{1}{(1 - z^2)^2}\Big((N + 3)(N + 2)U_{N - 2}(z) \\
& & - 2(N + 3)(N - 1)U_N(z) + N(N - 1)U_{N + 2}(z)\Big);
\end{eqnarray*}

\item $s = 3$,
\begin{eqnarray*}
U_N^{(3)}(z) &=& \frac{1}{8}\frac{1}{(1 - z^2)^3}\Big((N + 2)(N + 3)(N + 4)U_{N - 3}(z) \\
& & - 3(N - 2)(N + 3)(N + 4)U_{N - 1}(z) \\
& & + 3(N - 2)(N - 1)(N + 4)U_{N + 1}(z) \\
& & - (N - 2)(N - 1)NU_{N + 3}(z)\Big);
\end{eqnarray*}



\item $s = N-2$,
$$
\sum_{j = 1}^{N}(-1)^j j(j - 1){2N-1\choose N-j} U_{2j - 2}(z) = (N-1)2^{2N-3}(2Nz^2-1)(1-z^2)^{N-2}.
$$
Above the formula $\displaystyle\sum_{j=0}^n a_j=\sum_{j=0}^n a_{n-j}$ was used. 

\item $s = N-1$,
$$
\sum_{j = 1}^{N}(-1)^{j - 1} j{2N\choose N-j} U_{2j - 1}(z) = N2^{2N-1}z(1-z^2)^{N-1}.
$$

\item $s = N$,
\begin{displaymath}
\sum_{j = 0}^{N} (-1)^j{2N+1\choose N-j}U_{2j}(z) = 2^{2N}(1 - z^2)^N.
\end{displaymath}

\end{enumerate}


\begin{corollary}\label{cor:2} The following formulas are valid:\\
$\displaystyle \frac{1}{z}\sum_{j = 0}^{N} (-1)^j(j+1) {2N + 2 \choose N - j}U_{2j + 1}(z) = (2N + 2)\sum_{j = 0}^{N} (-1)^j{2N + 1 \choose N - j}U_{2j}(z)$, \\
or
$\displaystyle \sum_{j = 0}^{N}\frac{(-1)^j}{(N - j)!(N + 2 + j)!}\Big((j + 1)U_{2j + 1}(z) - (N + 2 + j)zU_{2j}(z)\Big) = 0$, \\
or
$\displaystyle \sum_{j = 0}^{N}\frac{(-1)^j}{(N - j)!(N + 2 + j)!}\Big((j + 1)U_{2j - 1}(z) + (N - j)zU_{2j}(z)\Big) = 0$.
\end{corollary} 

In \cite{ref13} the problem of constructing a non-trivial polynomial $\Psi\pa{\xi_0,\ldots,\xi_n}$ which vanishes when $\xi_j = U_j(z)$, $j = 1, \ldots, n$ is considered.
The formulas from Corollary \ref{cor:2} give a simpler solution to this problem, the desired polynomial is the sum of a linear combination of Chebyshev polynomials of the second kind with odd indices and the product of the polynomial $U_1(z)$ to a linear combination of Chebyshev polynomials of the second kind with even indices.

\bibliography{references}
\bibliographystyle{amsplain}

\end{document}